\documentclass{amsart}
\usepackage{amsmath}
\usepackage{amssymb}
\usepackage{amsthm}
\usepackage{amsfonts}
\usepackage{graphicx}
\newtheorem{Theorem}{Theorem}[section]

\newtheorem{Corollary}[Theorem]{Corollary}
\newtheorem{Lemma}[Theorem]{Lemma}
\newtheorem{Remark}[Theorem]{Remark}
\theoremstyle{definition}
\newtheorem{example}[Theorem]{Example}

\def \qed{\hfill{\hbox{$\square$}}}

\numberwithin{equation}{section}

\begin{document}
\title{Ricci Solitons on Pseudo--Riemannian Hypersurfaces of 4--dimensional Minkowski space}

\author[B. Bekta\c s]{Burcu Bekta\c s Demirci}
\address{Fatih Sultan Vak{\i}f University, Hal\.{I}\c{c} Campus, Faculty of Engineering,
Department of Civil Engineering, 34445, Beyo\u{g}lu, İstanbul}
\email{bbektas@fsm.edu.tr}

\date{}

\maketitle
\begin{abstract}
In this article, we get classification theorems for a Ricci soliton on the pseudo--Riemannian hypersurface 
of the Minkowski space 
$\mathbb{E}^4_1$ taking the potential vector field as the tangent component of the position vector of the pseudo--Riemannian hypersurface, denoted by $(M,g,{\bf x}^T,\lambda)$ in both Riemannian and Lorentzian settings.
First, we obtain the necessary and sufficient condition that a pseudo--Riemannian hypersurface $(M,g)$ in $\mathbb{E}^4_1$
admits a Ricci soliton $(M,g,{\bf x}^T,\lambda)$. 
In each of the form of the shape operator of a pseudo--Riemannian hypersurface, we obtain characterization
a Ricci soliton on a pseudo--Riemannian hypersurface. 
More precisely, we show that totally umbilical hypersurfaces, hyperbolic and a pseudo--spherical cylinder in $\mathbb{E}^4_1$
is a shrinking Ricci soliton whose the potential vector field is the tangent part of the position vector. 
Furthermore, we conclude that there exists only a shrinking Ricci soliton on a Lorentzian isoparametric hypersurface in $\mathbb{E}^4_1$ with nondiagonalizable shape operator whose the minimal polynomial has double real roots.   
\end{abstract}

\section{Introduction}
In the late of twentieth century, the concept of Ricci soliton was introduced by Hamilton 
to prove the Poincare Conjecture which is related to classify all compact three dimensional 
manifolds, \cite{Hamilton}.
The notion of Ricci soliton 
is related to self similar solutions of the Ricci flow, that is a partial differential equation
$\frac{\partial g(t)}{\partial t}=-2\mbox{Ric}(g(t))$ and they
often arise as limits of dilations of singularities in the Ricci flow. For more details, see \cite{Cao1, Cao2}. 
The topic of Ricci soliton has become more popular and important between mathematicians 
after G. Perelman used the notion of Ricci soliton to solve the Poincare conjecture in \cite{Perelman}. 

We recall the definition of a Ricci soliton on a pseudo--Riemannian submanifold as follows:\\
A smooth vector field $\xi$ on a pseudo--Riemannian manifold $(M,g)$ is said to define a Ricci soliton
if it satisfies
\begin{equation}
\label{riccisol}
\frac{1}{2}\mathcal{L}_\xi g+\mbox{Ric}=\lambda g
\end{equation}
where $\xi$ is a potential field, $\mathcal{L}_\xi g$ is the Lie deriative of the metric tensor $g$ with respect to $\xi$,
$\mbox{Ric}$ is the Ricci tensor of $(M,g)$ and $\lambda$ is a constant. 
In this context, the Ricci soliton can be considered as a natural generalization of the Einstein metric.
Throughout this work, a Ricci soliton on the pseudo--Riemannian manifold $(M,g)$ 
with a potential vector field $\xi$ is denoted by $(M,g,\xi,\lambda)$.
A Ricci soliton $(M,g,\xi,\lambda)$ is called shrinking, steady or expanding if
$\lambda>0$, $\lambda=0$ or $\lambda<0$, respectively.
A trivial Ricci soliton is one for which $\xi$ is zero or Killing, i.e., $\mathcal{L}_{\xi}g=0$ 
in which case the metric becomes Einstein. 
A Ricci soliton $(M,g,\xi, \lambda)$ is called a gradient Ricci soliton 
if its potential vector field $\xi$ is the gradient some smooth function $f$ on $M$,
that is, $\nabla f=\xi$.

Many mathematicians have obtained important results about the geometry of Ricci soliton from the different 
point of view such as \cite{ALG, FG, CK, BCGG, Chen2, Chen5, Chen6}.

The position vector of a pseudo--Riemannian submanifold in the pseudo--Euclidean space
is the most elementary and geometric object. There are some research fields in differential
geometry related to the position vector, see \cite{Chen3, Chen4}. 
A. Fialkow introduced in \cite{Fialkow} the notion of concircular vector fields $v$ on a Riemannian 
manifold $M$ as vector fields which satisfy 
\begin{equation}
\tilde{\nabla}_X v=\mu X,\;\;X\in TM
\end{equation}
where $\tilde{\nabla}$ denotes the Levi--Civita connection of $M$, 
$TM$ is the tangent bundle of $M$ 
and $\mu$ a nontrivial function on $M$. 
As a particular case for the function $\mu=1$, a concircular vector 
field $v$ is called a concurrent vector field. The notion of 
concircular vector fields can be extended naturally to concircular vector fields
in pseudo--Riemannian manifolds.  
Since $\tilde{\nabla}_X{\bf x}=X$ on a pseudo--Riemannian manifold $M$,
the position vector ${\bf x}$ of $M$ is the best known example as concurrent vector field. 

A. Barros et. al. studied the immersions of a Ricci soliton into a Riemannian manifold 
and they proved that a shrinking Ricci soliton immersed into a space form with
constant mean curvature is a Gaussian soliton in \cite{BGR}.
In \cite{Chen1} and \cite{Chen7}, B.Y.Chen and S. Deskmukh got the equation 
for the Ricci tensor of a submanifold in a Riemannian manifold to admit a Ricci soliton 
whose the potential vector field is the tangential part of concurrent vector field
and they also completely classfied Ricci solitons on Euclidean hypersurface whose 
the potential vector field arisen from the position vector field of Euclidean hypersurface.
Also, H. Al--Sodias et. al. obtained necessary and sufficient condition for a hypersurface in Euclidean space 
to be a gradient Ricci soliton in \cite{AAD}. 
Moreover, \c{S}. E. Meri\c{c} and E. K{\i}l{\i}\c{c} considered under which condition a submanifold of a Ricci soliton is also 
a Ricci soliton and they gave the relation between intrinsic and extrinsic invariants of a Riemannian submanifold
which admits a Ricci soliton in \cite{MK}. 

In this work, we study on a pseudo--Riemannian hypersurface 
$M$ in a 4--dimensional Minkowski space $\mathbb{E}^4_1$ whose the potential 
vector field is the tangential component of the position vector field
of a pseudo--Riemannian hypersurface $M$ in $\mathbb{E}^4_1$. 
First, we find the equation which the Ricci tensor of the pseudo--Rieamannian
hypersurfaces $M$ satisfies to be Ricci soliton. 
From the result in \cite{Magid}, we also know that the shape operator of a pseudo--Riemannian 
hypersurface $M$ in $\mathbb{E}^4_1$ 
can be put into the four different forms, not only diagonalizable one as Riemannian case. 
In this context, we classify the pseudo--Riemannian hypersurfaces in $\mathbb{E}^4_1$ with a diagonalizable shape
operator and constant mean curvature which admits a Ricci soliton $(M,g,{\bf x}^T,\lambda)$. 
Then, we study on a Lorentzian hypersurface $M$ in $\mathbb{E}^4_1$ whose the shape operator is not
diagonalizable. We show that there does not exist a Ricci soliton on a Lorentzian hypersurface
in $\mathbb{E}^4_1$ whose the minimal polynomial of the shape operator has complex principal curvature 
or real principal curvatures with multiplicity three. 
Finally, we obtain that a generalized umbilical Lorentzian hypersurface in $\mathbb{E}^4_1$
admits a shrinking Ricci soliton $(M, g, {\bf x}^T, \lambda)$. Thus, we get classification 
results for a Ricci soliton on an isoparametric hypersurface of $\mathbb{E}^4_1$.

\section{Preliminaries}
Let $\mathbb{E}^4_1$ be the $4$--dimensional real vector space 
$\mathbb{R}^4$ with a canonical pseudo--Euclidean metric tensor $\tilde{g}$ of index $1$ 
given by
\begin{equation}
\label{metric}
\tilde{g}=-dx_1^2+dx_2^2+dx_3^2+dx_4^2
\end{equation}
where $(x_1, x_2, x_3, x_4)$  is a rectangular coordinate system in $\mathbb{R}^4$.
Then, $\mathbb{E}^4_1=(\mathbb{R}^4,\tilde{g})$ is called a 4--dimensional Minkowski space. 

The pseudo--Riemannian space forms in $\mathbb{E}^4_1$ are defined by 
\begin{align}
\mathbb{S}^n_t(c^2)&=\{{\bf x}\in\mathbb{E}^4_1\;:\;\tilde{g}({\bf x},{\bf x})=c^{-2}\}\\
\mathbb{H}^n(-c^2)&=\{{\bf x}\in\mathbb{E}^4_1\;:\;\tilde{g}({\bf x},{\bf x})=-c^{-2}\}.
\end{align}
These spaces are complete and of constant curvature $c^2$ and $-c^2$, respectively. 
$\mathbb{S}^n_t(c^2)$ and $\mathbb{H}^n(-c^2)$ are called a pseudo--sphere and a hyperbolic 
space, respectively.

Let ${\bf x}:M\longrightarrow\mathbb{E}^4_1$ be an isometric immersion from 
a pseudo--Riemannian hypersurface $(M, g)$ to a pseudo--Euclidean space $(\mathbb{E}^4_1,\tilde{g})$. 
We denote the Levi--Civita connections of $\mathbb{E}^4_1$
and $M$ by $\tilde{\nabla}$ and $\nabla$, respectively. 
The fundamental formulas Gauss and Weingarten of the pseudo--Riemannian hypersurface $M$ in $\mathbb{E}^4_1$ 
are given by
\begin{equation}
\label{GW}
\tilde{\nabla}_XY=\nabla_XY+\varepsilon g(AX,Y)N
\;\;\mbox{and}\;\;
\tilde{\nabla}_X N=-AX
\end{equation}
where $N$ is the unit normal vector field with
$\varepsilon=\tilde{g}(N,N),\;\varepsilon=\pm 1,$ 
and $A$ is the shape operator of the pseudo--Riemannian hypersurface $M$ of $\mathbb{E}^4_1$
in the direction $N$.
The shape operator is a self--adjoint endomorphism of the tangent space of $M$, 
and the shape operator and the second fundamental form are related by 
\begin{equation}
\tilde{g}(h(X,Y),N)=g(AX,Y).
\end{equation} 
The hypersurface $M$ is called totally umbilical if and only if we have
\begin{equation}
h(X,Y)=g(X,Y)H
\end{equation}
for $X,Y\in T_p M$.
For $\varepsilon=1$, $M$ is a Lorentzian hypersurface of $\mathbb{E}^4_1$, that is,
the induced metric $g$ of $M$ in $\mathbb{E}^4_1$ is Lorentzian.
Otherwise, $M$ is a spacelike hypersurface of $\mathbb{E}^4_1$, that is,
the induced metric $g$ of $M$ in $\mathbb{E}^4_1$ is Riemannian.

From the Gauss equation, we get the Ricci tensor $\mbox{Ric}$ of 
the pseudo--Riemannian hypersurface $M$ in $\mathbb{E}^4_1$
\begin{equation}
\label{Riccieq}
\mbox{Ric}(X,Y)=3H g(AX,Y)-g(AX,AY)
\end{equation}
where $H=\frac{1}{3}\mbox{tr}A$ is the mean curvature function. 
Also, the Codazzi equation for the pseudo--Riemannian hypersurface $M$
of $\mathbb{E}^{4}_1$ is given by 
\begin{equation}
\label{Codazzieq}
(\nabla A)(X,Y)=(\nabla A)(Y,X),\;\;X,Y\in T_p(M)
\end{equation}
where  
$(\nabla A)(X,Y)=\nabla_X(AY)-A(\nabla_X Y)$.  

The gradient of the smooth function $f$ on $M$ is the vector field defined by 
\begin{equation}
\label{grad}
g(\nabla f,X)=df(X)=X(f)
\end{equation}
for any tangent vector field $X$ to $M$ and the Lie derivative of the metric $g$ is defined by
\begin{equation}
\label{Lie}
\mathcal{L}_{X}g(Y,Z)=X(g(Y,Z))-g([X,Y],Z)-g(Y,[X,Z])
\end{equation}
for any vector fields $X,Y,Z$ on $M$.
From now on, we will say $\{e_1,e_2,e_3\}$ is an orthonormal frame field defined on $M$ if 
$g(e_1,e_1)=-\varepsilon, g(e_2,e_2)=g(e_3,e_3)=1$ and $g(e_i,e_j)=0$ for $i\neq j$ and $i,j=1,2,3$.
On the other hand, we will say $\{e_1,e_2,e_3\}$ is a pseudo--orthonormal frame field defined on $M$ if 
$g(e_1,e_1)=g(e_2,e_2)=g(e_1,e_3)=g(e_2,e_3)=0$, $g(e_1,e_2)=-1$ and $g(e_3,e_3)=1$. 

\section{A Ricci soliton on a Pseudo--Riemannian Hypersurface of $\mathbb{E}^4_1$}
In this section, we give some following key results about the Ricci soliton on a pseudo--Riemannian hypersurface of $\mathbb{E}^4_1$ whose the
potential vector is the tangential part of the position vector.

\begin{Lemma}
Let  ${\bf x}:(M,g)\longrightarrow(\mathbb{E}^4_1,\tilde{g})$ be an isometric immersion from 
a pseudo--Riemannian hypersurface $M$ to 
$\mathbb{E}^4_1$. 
Then, the following equations are satisfied
\begin{equation}
\label{pos1}
\nabla_X {\bf x}^T=X+\varepsilon\rho AX \;\;\mbox{and}\;\;
\nabla{\rho}=-A{\bf x}^{T}
\end{equation}
where $\nabla{\rho}$ is the gradient of the smooth function $\rho$
defined by $\rho=\tilde{g}({\bf x},N)$.
\end{Lemma}

\begin{proof}
Let ${\bf x}:M\longrightarrow\mathbb{E}^4_1$
be an isometric immersion from a pseudo--Riemmanian hypersurface $(M,g)$
to a Minkowski space $(\mathbb{E}^4_1,\tilde{g})$.
We know that the position vector ${\bf x}$ can be decomposed as
\begin{equation}
\label{eq1}
 {\bf x}={\bf x}^{T}+\varepsilon\rho N   
\end{equation}
where ${\bf x}^{T}$ is the tangential component of ${\bf x}$ and 
$\rho=\tilde{g}({\bf x}, N)$ is a smooth function.

Considering ${\bf x}$ is the concurrent vector field on $\mathbb{E}^4_1$, i.e.
$\tilde{\nabla}_X{\bf x}=X$ for any tangent vector field $X$ to $M$,
and the equation \eqref{eq1}, 
we have
\begin{align}
\label{leq1}
X=\tilde{\nabla}_X{\bf x}
=\tilde{\nabla}_X({\bf x}^{T}+\varepsilon\rho N)
=\tilde{\nabla}_X{\bf x}^{T}
+\varepsilon(X(\rho)N+\rho\tilde{\nabla}_X N).
\end{align}
Using the Gauss and Weingarten formulas given in \eqref{GW},
we obtain the following one:
\begin{align}
\label{leq2}
X=\nabla_X{\bf x}^{T}+\varepsilon g(AX,{\bf x}^{T})N
+\varepsilon(X(\rho)N-\rho AX).
\end{align}
Then, comparing the tangential and normal part of the equation \eqref{leq2}, we get
\begin{align}
\nabla_X{\bf x}^{T}=X+\varepsilon\rho AX\;\;\mbox{and}\;\;
X(\rho)=- g(AX,{\bf x}^{T})
\end{align}
for any tangent vector field $X$ to $M$.\qed
\end{proof}

Now, we get the necessary and sufficent condition to be Ricci soliton $(M,g,{\bf x}^T,\lambda)$ on the pseudo--Riemannian hypersurface $M$ in $\mathbb{E}^4_1$.  

\begin{Theorem}
\label{Thm1}
A pseudo--Riemannian hypersurface $M$ in the Minkowksi space 
$\mathbb{E}^4_1$ admits a Ricci soliton
$(M, g, {\bf x}^{T}, \lambda)$ if and only if
the Ricci tensor $\mbox{Ric}$ of $(M, g)$ satisfies the following equation:
\begin{equation}
\label{Ricci1}
\mbox{Ric}(X,Y)=(\lambda-1)g(X,Y)-\varepsilon\rho g(AX,Y).
\end{equation}
Moreover, $(M, g, {\bf x}^{T}, \lambda)$ is a gradient Ricci soliton.
\end{Theorem}

\begin{proof}
Suppose that ${\bf x}:(M,g)\longrightarrow(\mathbb{E}^4_1,\tilde{g})$ is an isometric immersion
from a pseudo--Riemannian hypersurface $M$ to $\mathbb{E}^4_1$.
Then, we recall ${\bf x}={\bf x}^{T}+\varepsilon\rho N$ as explained above.
By the virtue of the equations \eqref{Lie} and \eqref{pos1}, 
we get the Lie derivative
$\mathcal{L}_{{\bf x}^T} g$ as
\begin{equation}
\left(\mathcal{L}_{{\bf x}^T} g\right) (X,Y)=
2 g(X,Y)+2\varepsilon\rho g(AX,Y)
\end{equation}
for any vector fields $X,Y$ and ${\bf x}^T$ on $M$.
Then, considering this equation with the definition of Ricci soliton
\eqref{riccisol}
we get the desired equation \eqref{Ricci1} easily.

Conversely, it can be immediately shown that a pseudo--Riemannian hypersurface 
$M$ in $\mathbb{E}^4_1$ satisfies the equation
\eqref{riccisol} with the potential vector field ${\bf x}^T$ and 
a constant $\lambda$ if the Ricci tensor of $M$ in $\mathbb{E}^4_1$ holds the given condition in \eqref{Ricci1}. 
Thus, $(M,g,\lambda, {\bf x}^T)$ is a Ricci soliton on a pseudo--Riemannian 
hypersurface $M$ of $\mathbb{E}^4_1$. 

To prove that $(M,g,\lambda, {\bf x}^T)$ is also a gradient Ricci soliton, 
we choose any smooth function $f$ as 
$f=\frac{1}{2}\tilde{g}({\bf x},{\bf x})$. 
Then, we compute
\begin{equation}
X(f)=\tilde{g}(\tilde{\nabla}_X{\bf x},{\bf x})=\tilde{g}(X,{\bf x})=g(X,{\bf x}^T)
\end{equation}
which implies $\nabla f={\bf x}^T$. 
Thus, $(M,g,\lambda, {\bf x}^T)$ is a gradient Ricci soliton on the pseudo--Riemannian 
hypersurface $M$ in $\mathbb{E}^4_1$.
\qed
\end{proof}

In \cite{Magid}, it was given that a symmetric endomorphism $A$ of a vector space $V$
with a Lorentzian inner product can be put into four different forms which give counterpart of Riemannian case. 
Since the shape operator $A$ of the pseudo--Riemannian 
hypersurface $M$ in $\mathbb{E}^4_1$ is a symmetric endomorphism, we consider 
the four following forms for the shape operator $A$ with respect to orthonormal and pseudo--orthonormal 
bases of $\mathbb{E}^4_1$.   

\subsection{A Ricci soliton on pseudo--Riemannian hypersurfaces of $\mathbb{E}^4_1$ with a diagonalizable shape operator}
In this section, we are going to consider the case the Ricci soliton on the pseudo--Riemannian hypersurface 
in $\mathbb{E}^4_1$ that the shape operator 
is diagonalizable. 

For using later, we recall the definition of isoparametric and a generalized constant ratio pseudo--Riemannian hypersurfaces 
in $\mathbb{E}^4_1$ as follows:
A pseudo--Riemannian hypersurface is said to be isoparametric if its shape operator 
is diagonalizable and has constant principal curvatures. 
A pseudo--Riemannian hypersurface is said to be a generalized constant ratio hypersurface if the tangential
part of its position vector is one of principal direction of the shape operator. 

Suppose that $M$ is a pseudo--Riemannian hypersurface in $\mathbb{E}^4_1$
with a diagonalizable shape operator $A$. 
Then, we consider an orthonormal tangent frame field 
$\{e_1, e_2, e_3\}$ on $M$ such that $Ae_i=a_ie_i,\;i=1,2,3,$ for smooth functions $a_1,a_2$ and $a_3$.
The connection forms corresponding to chosen tangent frame field is defined by 
$w_{ij}(e_k)=\varepsilon_j g(\tilde{\nabla}_{e_k}{e_i},e_j)$.
Also, from the Codazzi equation \eqref{Codazzieq}, we get 
\begin{align}
\begin{split}
\label{codazzid1}
e_i(a_j)&=g(e_j,e_j)\omega_{ij}(e_j)(a_i-a_j)\\
\omega_{ij}(e_k)(a_i-a_j)&=\omega_{ik}(e_j)(a_i-a_k)
\end{split}
\end{align}
where indices $i,j,k$ are distinct and $i,j,k=1,2,3$.

On the other hand, the equation \eqref{Riccieq} gives us the components of Ricci tensor $\mbox{Ric}$ 
of $M$ corresponding to chosen frame field as follows:
\begin{align}
\label{ricI}
\begin{split}
\mbox{Ric}(e_1,e_1)&=-\varepsilon a_1(a_2+a_3),\qquad
\mbox{Ric}(e_2,e_2)=a_2(a_1+a_3),\\
\mbox{Ric}(e_3,e_3)&=a_3(a_1+a_2),\qquad\;\;
\mbox{Ric}(e_i,e_j)=0,\;\; i\neq j
\end{split}
\end{align}
for $i,j=1,2,3$. 
Now, we determine whether there exists a Ricci soliton 
$(M,g,{\bf x}^T,\lambda)$ on the pseudo--Riemannian hypersurface $M$ in $\mathbb{E}^4_1$ 
whose the shape operator is a diagonalizable one.
\begin{Lemma}
\label{lemma1}
Let ${\bf x}:(M,g)\longrightarrow(\mathbb{E}^4_1,\tilde{g})$ be an isometric immersion from a pseudo--Riemannian hypersurface $M$ to the Minkowski space $\mathbb{E}^4_1$ with a diagonalizable the shape operator $A$ such as 
\begin{equation}
    A=\mbox{diag}(a_1,a_2,a_3)
\end{equation}
for some smooth functions $a_1,a_2,a_3$. 
Then, $(M,g,{\bf x}^T, \lambda)$ is a Ricci soliton if and only if 
the shape operator $A$ of $M$ satisfies one of the followings:
\begin{itemize}
    \item [(i.)] $A=c\bf{I}$ for a nonzero constant $c$ and $3\times 3$ identity matrix ${\bf I}$.  
    
    \item [(ii.)] $A$ has two distinct principal curvatures.
\end{itemize} 
\end{Lemma}

\begin{proof}
Suppose that a pseudo--Riemannian hypersurface $M$ of $\mathbb{E}^4_1$ has  
a diagonalizable shape operator $A$. Thus, we have an orthonormal tangent frame field 
$\{e_1, e_2, e_3\}$ on $M$ such that $Ae_i=a_ie_i,\;i=1,2,3,$ for smooth functions $a_1,a_2$ and $a_3$.
For the chosen frame field, the equations 
\eqref{codazzid1} and \eqref{ricI} hold. 
Since $(M,g,{\bf x}^T, \lambda)$ is Ricci soliton, 
from the equation \eqref{Ricci1} in Theorem \ref{thm1}, 
we calculate the components of Ricci tensor $\mbox{Ric}$ of $M$ in $\mathbb{E}^4_1$ as follows:
\begin{align}
\label{riccI}
\begin{split}
\mbox{Ric}(e_1,e_1)&=\varepsilon(1-\lambda)+\rho a_1,\qquad
\mbox{Ric}(e_2,e_2)=\lambda-1-\varepsilon\rho a_2,\\
\mbox{Ric}(e_3,e_3)&=\lambda-1-\varepsilon\rho a_3,\qquad\;\;
\mbox{Ric}(e_i,e_j)=0,\;\; i\neq j
\end{split}
\end{align}
for $i,j=1,2,3$. 
Comparing equations \eqref{ricI} and \eqref{riccI}, 
we have the following system of equation:
\begin{align}
\label{eq1I}
\lambda-1-\varepsilon\rho a_1&=a_1(a_2+a_3),\\
\label{eq2I}
\lambda-1-\varepsilon\rho a_2&=a_2(a_1+a_3),\\
\label{eq3I}
\lambda-1-\varepsilon\rho a_3&=a_3(a_1+a_2).
\end{align}
Hence, we find the following two equations 
\begin{equation}
\label{cond1}
(a_1-a_2)(a_3+\varepsilon\rho)=0\:\:\mbox{and}\;\; 
(a_1-a_3)(a_2+\varepsilon\rho)=0
\end{equation}
which give two different solutions given as follows.

If $a_1=a_2=a_3$, that is, $M$ is a totally umbilical pseudo--Riemannian hypersurface in $\mathbb{E}^4_1$, 
the equations in \eqref{cond1} are satisfied trivially. 
Also, from the equation \eqref{codazzid1}, we find that all $a_i$'s are non--zero constant.
Remark that when $a_1=a_2=a_3=0$, the equation \eqref{ricI} gives $\mbox{Ric}(e_i,e_j)=0$ for $i,j=1,2,3$. 
We omit this case. Thus, we obtain the desired result (i). Moreover, $\lambda=2c^2+\varepsilon\rho c$.

$a_1=a_2=-\varepsilon\rho\neq a_3$ or $a_1=a_3=-\varepsilon\rho\neq a_2$ 
are also the solutions of the given system in \eqref{cond1}.  
Thus, we say that the shape operator $A$ of the pseudo--Riemannian hypersurface $M$
in $\mathbb{E}^4_1$ has two distinct principal curvatures. Moreover, $\lambda=1+a_1a_3$ or $\lambda=1+a_2a_3$.

Conversely, it can be easily seen that the Ricci tensor $\mbox{Ric}$ 
of the pseudo--Riemannian hypersurface $M$ in $\mathbb{E}^4_1$ 
whose the shape operator $A$ is one of the cases (i) and (ii)
satisfies the equation \eqref{Ricci1}. 
Thus, $(M,g,{\bf x}^T,\lambda)$ is a Ricci soliton 
on such a pseudo--Riemannian hypersurface $M$ of $\mathbb{E}^4_1$.
\qed
\end{proof}

In \cite{AKY}, N. Abe et. al studied the pseudo--Riemannian hypersurfaces in a pseudo--Riemannian
real space forms and they got a congruence theorem for isoparametric pseudo--Riemannian hypersurfaces whose 
shape operators have at most two mutually distinct constant principal curvatures. 
From the result in \cite{AKY}, we consider the following examples of isoparametric pseudo--Riemannian hypersurfaces 
in the Minkowski space $\mathbb{E}^4_1$: 
\begin{example}
\label{ex1}
A hyperbolic space $\mathbb{H}^3(-c^2)$ with constant curvature $-c^2$ defined by 
$$\mathbb{H}^3(-c^2)=\left\{{\bf x}=(x_1,x_2,x_3,x_4)\in\mathbb{E}^4_1\;|\;-x_1^2+x_2^2+x_3^2+x_4^2=-c^{-2}\right\}$$
and de Sitter space $\mathbb{S}^3_1(c^2)$ with constant curvature $c^2$ defined by 
$$\mathbb{S}^3_1(c^2)=\left\{{\bf x}=(x_1,x_2,x_3,x_4)\in\mathbb{E}^4_1\;|\;-x_1^2+x_2^2+x_3^2+x_4^2={c^{-2}}\right\}$$
are totally umbilical isoparametric spacelike and Lorentzian hypersurfaces in $\mathbb{E}^4_1$ with 
the shape operator $A=c{\bf I}$ for $3\times 3$ identity matrix ${\bf I}$, respectively. 
From Theorem \ref{thm1}, it can be easily seen that the Ricci tensor of 
$\mathbb{H}^3(-c^2)$ and $\mathbb{S}^3_1(c^2)$ in $\mathbb{E}^4_1$ 
satisfy \eqref{Ricci1} with $\lambda=c^2$ or $\lambda=3c^2$. 
Thus, $(\mathbb{H}^3(-c^2),g,{\bf x}^T,\lambda)$ and $(\mathbb{S}^3_1(c^2),g,{\bf x}^T,\lambda)$ 
are a shrinking Ricci soliton in $\mathbb{E}^4_1$.  
\end{example}

\begin{example}
\label{ex3}
Let $\mathbb{H}^2(-c^2)\times\mathbb{E}$ in the Minkowski space $\mathbb{E}^4_1$ be a spacelike hyperbolic cylinder 
defined by 
\begin{equation}
\mathbb{H}^2(-c^2)\times\mathbb{E}=\left\{{\bf x}=(x_1,x_2,x_3,x_4)\in\mathbb{E}^4_1\;:\;
(x_1,x_2,x_3)\in\mathbb{E}^3_1\;\mbox{and}\; -x_1^2+x_2^2+x_3^2=-{c^{-2}}\right\}
\end{equation}
for a nonzero constant $c$. 
The shape operator $A$ has principal curvatures $c, c$ and $0$. 
From Theorem \ref{thm1}, it can be easily seen that the Ricci tensor of 
the spacelike hyperbolic cylinder $\mathbb{H}^2(-c^2)\times\mathbb{E}$ in $\mathbb{E}^4_1$ 
satisfies \eqref{Ricci1} with $\lambda=1$. 
Hence, $(\mathbb{H}^2(-c^2)\times\mathbb{E},g,{\bf x}^T,\lambda)$ is a shrinking Ricci soliton in $\mathbb{E}^4_1$. 
\end{example}

\begin{example}
\label{ex2}
Let $\mathbb{S}^2_1(c^2)\times\mathbb{E}$ in the Minkowski space $\mathbb{E}^4_1$ be a pseudo--spherical 
cylinder parametrized as 
\begin{equation}
\mathbb{S}^2_1(c^2)\times\mathbb{E}=\left\{{\bf x}=(x_1,x_2,x_3,x_4)\in\mathbb{E}^4_1\;:\;
(x_1,x_2,x_3)\in\mathbb{E}^3_1\;\mbox{and}\; -x_1^2+x_2^2+x_3^2={c^{-2}}\right\}
\end{equation}
for a nonzero constant $c$. 
The shape operator $A$ has principal curvatures $c, c$ and $0$. 
From Theorem \ref{thm1}, it can be easily seen that 
the a pseudo--spherical 
cylinder $\mathbb{S}^2_1(c^2)\times\mathbb{E}$ in $\mathbb{E}^4_1$ 
satisfies \eqref{Ricci1} with $\lambda=1$. 
Hence, $(\mathbb{S}^2_1(c^2)\times\mathbb{E},g,{\bf x}^T,\lambda)$ is a shrinking Ricci soliton in $\mathbb{E}^4_1$.
\end{example}

By using Theorem \ref{thm1}, we get a classification theorem for the Ricci soliton on 
the pseudo--Riemannian hypersurface in $\mathbb{E}^4_1$ with a diagonalizable shape operator
and constant mean curvature as follows.

\begin{Theorem}
\label{thm1}
Let $M$ be a pseudo--Riemannian hypersurface of the Minkowski space $\mathbb{E}^4_1$ 
with a diagonalizable the shape operator and constant mean curvature. 
Then, $(M,g,{\bf x}^T, \lambda)$ is a Ricci soliton if and only if
the pseudo--Riemannian hypersurface $M$ is an open portion of one of the followings:
\begin{itemize}
\item [(i.)] a hyperbolic space $\mathbb{H}^3(-c^2)\subset\mathbb{E}^4_1$ given in Example \ref{ex1},

\item [(ii.)] de Sitter space $\mathbb{S}^3_1(c^2)\subset\mathbb{E}^4_1$ given in Example \ref{ex1},

\item [(iii.)] a hyperbolic cylinder $\mathbb{H}^2(-c^2)\times\mathbb{E}$ given in Example \ref{ex3},

\item [(iv.)]  a pseudo--spherical cylinder $\mathbb{S}^2_1(c^2)\times\mathbb{E}$ given in Example \ref{ex2},
\end{itemize}
\end{Theorem}

\begin{proof}
Assume that  $(M,g,{\bf x}^T, \lambda)$ is a Ricci soliton on the pseudo--Riemannian
hypersurface $M$ of $\mathbb{E}^4_1$ with a diagonalizable shape operator $A=\mbox{diag}(a_1,a_2,a_3)$ 
for some smooth functions $a_1, a_2, a_3$ and the mean curvature $H=\frac{1}{3}(a_1+a_2+a_3)$ is a constant .  
From Lemma \ref{lemma1}, there are two following cases occur for the shape operator $A$.\\
\textit{Case a.} 
Suppose that there exists an orthonormal tangent frame field $\{e_1, e_2, e_3\}$ on $M$ such that 
$Ae_i=ce_i$ for a nonzero constant $c$. That is, 
$M$ is a totally umbilical isoparametric pseudo--Riemannian hypersurface of $\mathbb{E}^4_1$. 
Also, it is a trivial result that $M$ has constant mean curvature $H=c$ in $\mathbb{E}^4_1$. 
Thus, from the result in \cite{AKY} given as Example \ref{ex1}, 
we get the pseudo--Riemannian hypersurfaces in $\mathbb{E}^4_1$ given by (i) and (ii) 
with respect to $\varepsilon=1$ or $\varepsilon=-1$.  
\\
\textit{Case b.} Suppose that the shape operator $A$ has two distinct principal curvatures. 
Without loss of generality, we take $a_1=a_2\neq a_3$.
Then, the second one of \eqref{cond1} yields
$a_1=a_2=-\varepsilon\rho$. 
Hence, we get $a_1a_3=\lambda-1$ by using the equation \eqref{eq1I}.
In such case, the equations \eqref{codazzid1} become the followings:
\begin{align}
\label{codazzid2_1}
e_1(a_1)&=e_2(a_1)=e_1(a_3)=e_2(a_3)=0,\\
\label{codazzid2_2}
\omega_{13}(e_2)&=\omega_{23}(e_1)
=\omega_{13}(e_3)=\omega_{23}(e_3)=0,\\
\label{codazzid2_3}
e_3(a_1)&=-\varepsilon\omega_{13}(e_1)(a_1-a_3),\\
\label{codazzid2_4}
a_1e_3(a_3)&=\varepsilon\omega_{13}(e_1)(a_1-a_3)a_3.
\end{align} 
Taking consideration $e_i(a_1)=0$ for $i=1,2$ and $a_1=-\varepsilon\rho$, 
we have $e_i(\rho)=0$ for $i=1,2$. Thus,
the equation \eqref{pos1} implies 
\begin{align}
\label{grad1}
    &a_1g(e_1,{\bf x}^T)=a_1g(e_2,{\bf x}^T)=0,\\
\label{grad2}    
    &e_3(\rho)=-a_3g(e_3,{\bf x}^T).
\end{align}
For $a_1=0$, $a_1a_3=\lambda-1$ implies $\lambda=1$. Also, we get $\rho=\tilde{g}({\bf x},N)=0$, 
that is, ${\bf x}={\bf x}^T$. 
On the other hand, from the equation \eqref{ricI}, 
it can be easily seen that
the Ricci tensor of $M$ in $\mathbb{E}^4_1$ equals zero. 
Thus, $a_1\neq 0$, that is, $g(e_1,{\bf x}^T)=g(e_2,{\bf x}^T)=0$.  
Hence, the position vector 
${\bf x}$ of $M$ in $\mathbb{E}^4_1$ is decomposed as 
\begin{equation}
\label{gcr}
{\bf x}=g({\bf x}^T,e_3)e_3+\varepsilon\rho N
\end{equation}
which implies that $M$ is a generalized constant ratio pseudo--Riemannian hypersurfaces in $\mathbb{E}^4_1$.
Since $M$ has constant mean curvature in $\mathbb{E}^4_1$,
we get $e_3(a_3)=-2e_3(a_1)$.
Considering the equations \eqref{codazzid2_3} and \eqref{codazzid2_4} with this equality, 
we have 
\begin{equation}
\varepsilon\omega_{13}(e_1)(a_1-a_3)\left(1-\frac{a_3}{2a_1}\right)=0.
\end{equation}
When $\omega_{13}(e_1)=0$, from the equations \eqref{codazzid2_3}, 
we obtain $a_1=-\varepsilon\rho$ is a nonzero constant. 
Thus, the equation \eqref{grad2} gives $a_3=0$ which means $\lambda=1$. 
Remark that for $g({\bf x}^T,e_3)=0$, we get ${\bf x}^T=0$. In this case,
we have a trivial Ricci soliton.   
For $a_3=2a_1$, $a_1a_3=2a_1^2=\lambda-1$ implies that $a_1=-\varepsilon\rho$ is a nonzero constant.
Thus, $a_3=0$ which implies $\lambda=1$ again. Hence, the shape operator $A$ has the principal curvatures 
$a_1=a_2=-\varepsilon\rho$ and $a_3=0$. Moreover, $\lambda=1$. 
In this case, $M$ is an isoparametric pseudo--Riemannian hypersurface in $\mathbb{E}^4_1$. 
Thus, from the result in \cite{AKY} given as Example \ref{ex2} and Example \ref{ex3}, 
we get the pseudo--Riemannian hypersurfaces given by (iii) and (iv) 
with respect to $\varepsilon=1$ or $\varepsilon=-1$ for $c=\varepsilon\rho$.\qed
\end{proof}

\begin{Corollary}
\label{cor1}
There exists only shrinking Ricci soliton on a pseudo--Riemannian hypersurface in $\mathbb{E}^4_1$ with a diagonalizable shape
operator and constant mean curvature. 
\end{Corollary}

\subsection{A Ricci soliton on Lorentzian hypersurfaces of $\mathbb{E}^4_1$ with nondiagonalizable shape operator} 
In this section, we will consider the Ricci soliton on the Lorentzian hypersurface of $\mathbb{E}^4_1$ having nondiagonalizable shape operator. 

From \cite{Magid}, we know that the shape operator $A$ of the Lorentzian hypersurface $M$ in $\mathbb{E}^4_1$ can be put one of the following form:
\begin{itemize}
    \item [(i.)] 
    $A=\left(
    \begin{array}{ccc}
    a_1 & b_1 & 0\\
    -b_1 & a_1 & 0\\
    0 & 0 & a_2
    \end{array}
    \right),
    $ 
    with respect to orthonomal frame field for $b_1\neq 0$.
    
    \item [(ii.)] 
    $A=\left(
    \begin{array}{ccc}
    a_1 & 0 & 0\\
    0 & a_1 & 1\\
    -1 & 0 & a_1
    \end{array}
    \right)
    $ 
    with respect to pseudo orthonomal frame field.
    
    \item [(iii.)] 
    $A=\left(
    \begin{array}{ccc}
    a_1 & 0 & 0\\
    1 & a_1 & 0\\
    0 & 0 & a_2
    \end{array}
    \right)
    $ 
    with respect to pseudo orthonomal frame field.
\end{itemize}
A Lorentzian hypersurface with nondiagonalizable shape operator is said to be isoparametric 
if the minimal polynomial of its shape operator is constant. The minimal polynomial of 
the shape operator $A$ is a monic polynomial $p$ of least degree such that $p(A)=0$. 
\begin{Theorem}
\label{thm2}
Let $M$ be a Lorentzian hypersurface in the Minkowski space $\mathbb{E}^4_1$
with a nondiagonalizable shape operator $A$.
Then, $(M,g,{\bf x}^T, \lambda)$ is a Ricci soliton if and only if 
there exists a pseudo--orthonormal frame field on $M$ to put the shape operator into the following form 
\begin{equation}
    A=\left(
    \begin{array}{ccc}
    a_1 & 0 & 0\\
    1 & a_1 & 0\\
    0 & 0 & a_1
    \end{array}
    \right).
\end{equation}
Moreover, $M$ is an isoparametric Lorentzian hypersurface of $\mathbb{E}^4_1$ 
and $(M,g,{\bf x}^T, \lambda)$ is a shrinking Ricci soliton. 
\end{Theorem}

\begin{proof}
Suppose that a Lorentzian hypersurface $M$ of $\mathbb{E}^4_1$ 
has a nondiagonalizable shape operator $A$. 
Thus, there are three cases occurs for the form of the shape operator $A$.

\textit{Case a.} $Ae_1=a_1e_1-b_1e_2,\; Ae_2=b_1e_1+a_1e_2$ and $Ae_3=a_2e_3$ with respect to an orthonormal frame field  
$\{e_1,e_2,e_3\}$ defined on $M$ for some functions $a_1, a_2$ and $b_1\neq 0$.
Then, 
the mean curvature of $M$ in $\mathbb{E}^4_1$ is 
${\it H}=\frac{1}{3}(2a_1+a_2).$
From the equation \eqref{Riccieq}, 
we get the components of Ricci tensor of a Lorentzian hypersurface $M$ as follows:
\begin{align}
\label{ricIV}
\begin{split}
&\mbox{Ric}(e_2,e_2)=-\mbox{Ric}(e_1,e_1)=a_1^2+a_1a_2+b_1^2,\\
&\mbox{Ric}(e_1,e_2)=-a_2b_1,\qquad
\mbox{Ric}(e_3,e_3)=2a_1a_2,\\
&\mbox{Ric}(e_1,e_3)=\mbox{Ric}(e_2,e_3)=0.
\end{split}
\end{align}
On the other hand, 
if $M$ admits a Ricci soliton $(M,g,{\bf x}^T,\lambda)$, 
then we have the components of Ricci tensor of $M$ from 
the equation \eqref{Ricci1} in Theorem \ref{Thm1} for 
$\varepsilon=1$ given by
\begin{align}
\label{riccIV}
\begin{split}
&\mbox{Ric}(e_1,e_1)=-\mbox{Ric}(e_2,e_2)=1-\lambda+\rho a_1,\\
&\mbox{Ric}(e_3,e_3)=\lambda-1-\rho a_2,\\
&\mbox{Ric}(e_1,e_2)=\rho b_1,\quad
\mbox{Ric}(e_1,e_3)=\mbox{Ric}(e_2,e_3)=0.
\end{split}
\end{align}
Considering equations \eqref{ricIV} and \eqref{riccIV} together,
we obtain the following system of equations:
\begin{align}
\label{eq1IV}
\lambda-1-\rho a_1&=a_1^2+a_1a_2+b_1^2,\\
\label{eq2IV}
\lambda-1-\rho a_2&=2a_1a_2,\\
\label{eq3IV}
\rho b_1&=-a_2b_1.
\end{align} 
Since $b_1\neq 0$, the equation \eqref{eq3IV} implies
$\rho=-a_2$. 
Then, from the equation \eqref{eq1IV}, 
we have $\lambda=1+a_1^2+b_1^2$.
Using all in the equation \eqref{eq2IV}, 
we get $(a_1-a_2)^2+b_1^2=0$ that gives $b_1=0$. 
This is a contradiction. Hence, 
a Lorentzian hypersurface $M$ in $\mathbb{E}^4_1$ having such a nondiagonalizable shape operator $A$ does not admit  
the Ricci soliton $(M,g,{\bf x}^T,\lambda)$. 

\textit{Case b.} $Ae_1=a_1e_1-e_3,\;  Ae_2=a_1e_2$ and $Ae_3=e_2+a_1e_3$ with respect to a pseudo--orthonormal frame 
field  $\{e_1,e_2,e_3\}$ defined on $M$ for function
$a_1$. 
Also, the mean curvature of $M$ in $\mathbb{E}^4_1$ 
is $H=a_1.$
From the equation \eqref{Riccieq}, we compute 
the coefficients of Ricci tensor of $M$ given by 
\begin{align}
\label{ricIII}
\begin{split}
&\mbox{Ric}(e_1,e_1)=-1,\qquad\qquad\qquad\qquad\; \;\;
\mbox{Ric}(e_1,e_3)=-a_1, \\
&\mbox{Ric}(e_3,e_3)=-\mbox{Ric}(e_1,e_2)=2a_1^2,\qquad
\mbox{Ric}(e_2,e_2)=\mbox{Ric}(e_2,e_3)=0.
\end{split}
\end{align}
If $(M,g,{\bf x}^T,\lambda)$ is a Ricci soliton
on $M$ of $\mathbb{E}^4_1$, then using the equation 
\eqref{Ricci1} for $\varepsilon=1$, 
we calculate the components 
of Ricci tensor of a Lorentzian hypersurface $M$ as follows:
\begin{align}
\label{riccIII}
\begin{split}
&\mbox{Ric}(e_1,e_1)=\mbox{Ric}(e_2,e_2)=\mbox{Ric}(e_2,e_3)=0,
\quad\mbox{Ric}(e_1,e_3)=\rho \\
&\mbox{Ric}(e_1,e_2)=-\mbox{Ric}(e_3,e_3)=1-\lambda+\rho a_1.\\
\end{split}
\end{align}
It can be easily seen that $\mbox{Ric}(e_1,e_1)=-1\neq 0= \mbox{Ric}(e_1,e_1)$ which gives 
a contradiction. 
Thus, there does not exist a Ricci soliton $(M, g, {\bf x}^T, \lambda)$ on   
a Lorentzian hypersurface $M$ of $\mathbb{E}^4_1$ having such a nondiagonalizable shape operator $A$.

\textit{Case c.} 
$Ae_1=a_1e_1+e_2,\; Ae_2=a_1e_2$ and $Ae_3=a_2e_3$ with respect to 
a pseudo--orthonormal frame field  
$\{e_1,e_2,e_3\}$ defined on $M$ for smooth functions $a_1$ and $a_2$. 
Then, the mean curvature of $M$ in $\mathbb{E}^4_1$ is 
$H=\frac{1}{3}(2a_1+a_2).$
By a direct calculation from the equation \eqref{Riccieq}, 
we obtain the components of Ricci tensor of $M$ given by 
\begin{align}
\label{riccII}
\begin{split}
\mbox{Ric}(e_1,e_1)&=-a_2,\qquad\qquad\;\;\;\;\;
\mbox{Ric}(e_1,e_3)=\mbox{Ric}(e_2,e_2)=\mbox{Ric}(e_2,e_3)=0\\
\mbox{Ric}(e_1,e_2)&=-a_1(a_1+a_2),\quad\;
\mbox{Ric}(e_3,e_3)=2a_1a_2.
\end{split}
\end{align}

Since $(M,g,{\bf x}^T,\lambda)$ is
a Ricci soliton on a Lorentzian hypersurface of $M$ of 
$\mathbb{E}^4_1$, that is $\varepsilon=1$, the equation \eqref{Ricci1} 
gives the following equations:
\begin{align}
\label{ricII}
\begin{split}
\mbox{Ric}(e_1,e_1)&=\rho,\qquad\qquad\;\;\;\;\;
\mbox{Ric}(e_1,e_3)=\mbox{Ric}(e_2,e_2)=\mbox{Ric}(e_2,e_3)=0\\
\mbox{Ric}(e_1,e_2)&=1-\lambda+\rho a_1,\quad
\mbox{Ric}(e_3,e_3)=\lambda-1-\rho a_2.
\end{split}
\end{align}
Considering equations \eqref{riccII} and \eqref{ricII}, 
we get $\rho=-a_2$ and the following equation system is also obtained 
\begin{align}
\label{eq1_II}
\lambda-1-\rho a_2&=2a_1a_2,\\
\label{eq2_II}
\lambda-1-\rho a_1&=a_1^2+a_1a_2.
\end{align}
which give $a_1=a_2$ and $\lambda=a_1^2+1>0$. 
Thus, $(M,g,{\bf x}^T,\lambda)$ is a shrinking Ricci soliton.   
Since $\lambda$ is a constant, $a_1$ is constant. Hence, $M$ is an isoparametric Lorentzian 
hypersurface of $\mathbb{E}^4_1$. \qed
\end{proof}
From the proof of Theorem \eqref{thm}, we say that there does not exist a Ricci soliton $(M,g,{\bf x}^T,\lambda)$
on a Lorentzian hypersurface $M$ in $\mathbb{E}^4_1$ whose the minimal polynomial of the shape operator 
has no complex roots or real roots with multiplicity with three. 

In \cite{Magid}, M. Magid studied a Lorentzian isoparametric hypersurfaces in $\mathbb{E}^4_1$ 
and they got classification theorem a Lorentzian hypersurfaces according to the principal curvatures
and the minimal polynomial of the shape operator. 
From the result in \cite{Magid}, we consider the following examples of Lorentzian isoparametric hypersurfaces 
in the Minkowski space $\mathbb{E}^4_1$: 
\begin{example}
\label{ex4}
Let $\alpha(s)$ be a null curve in $\mathbb{E}^4_1$ with a pseudo--orthonormal frame 
$\{X(s), Y(s), Z(s), W(s)\}$ such that 
\begin{align}
\begin{split}
\dot{\alpha}(s)&=X(s),\\
\dot{Z}(s)&=-aX(s)-B(s)Y(s)
\end{split}
\end{align}
where $B(s)\neq 0$ and $a$ is a nonzero constant. 
The Lorentzian hypersurface $M$ in $\mathbb{E}^4_1$ is locally defined 
as follows:
\begin{equation}
\label{guh1}
M: {\bf x}(s,u,v)=\alpha(s)+uY(s)+vW(s)
+\left(\sqrt{\frac{1}{a^2}-v^2}-\frac{1}{a}\right)Z(s)
\end{equation}
which is called as a generalized umbilical hypersurface.
The unit normal vector $N$ is 
\begin{equation}
\label{guh1n}
N(s,u,v)=-auY(s)-\sqrt{1-a^2v^2}Z(s)-avW(s).
\end{equation}
and the shape operator $A$ of $M$ in the direction $N$ 
can be put in the following form with respect to a suitable chosen pseudo--orthonormal frame field
\begin{equation}
\label{guh1s}
A=
\left[\begin{array}{ccc}
a&0&0\\
k(s)&a&0\\
0&0&a
\end{array}
\right],\;\;k(s)\neq 0.
\end{equation}
The minimal polynomial of $A$ is $(x-a)^2$.  
From Theorem \ref{thm2}, the Ricci tensor of the generalized umbilical Lorentzian hypersurface $M$
in $\mathbb{E}^4_1$ given by \eqref{guh1}
satisfies \eqref{Ricci1} with $\lambda=a^2+1$. 
Hence, $(M,g,{\bf x}^T,\lambda)$ is a shrinking Ricci soliton in $\mathbb{E}^4_1$.
\end{example} 

\begin{example}
\label{ex5}
Let $\alpha(s)$ be a null curve in $\mathbb{E}^4_1$ with a pseudo--orthonormal frame 
$\{X(s), Y(s), Z(s), W(s)\}$ such that 
\begin{align}
\begin{split}
\dot{\alpha}(s)&=X(s),\\
\dot{Z}(s)&=-B(s)Y(s)\;\;B(s)\neq 0
\end{split}
\end{align} 
The Lorentzian hypersurface $M$ in $\mathbb{E}^4_1$ is locally defined 
as follows:
\begin{equation}
\label{guh2}
M: {\bf x}(s,u,v)=\alpha(s)+uY(s)+vW(s)
\end{equation}
which is called as a generalized cylinder of type I.
The unit normal vector $N$ is $N=Z(s)$ 
and the shape operator $A$ of $M$ in the direction $N$ 
can be put in the following form with respect to a suitable chosen pseudo--orthonormal frame field on $M$
\begin{equation}
\label{guh2s}
A=
\left[\begin{array}{ccc}
0&0&0\\
k(s)&0&0\\
0&0&0
\end{array}
\right],\;\; k(s)\neq 0.
\end{equation}
Also, we have that the minimal polynomial of $A$ is $x^2$. 
From a direct calculation, it can be easily seen that the Ricci tensor $\mbox{Ric}$ of hypersurface $M$
in $\mathbb{E}^4_1$ equals zero, that is, $\mathcal{L}_{{\bf x}^T}=g$. 
Thus, $M$ has a conformal killing vector field ${\bf x}^T$
\end{example} 

From the proof of Theorem \ref{thm2}, 
we obtain that the minimal polynomial of $A$ is $p(x)=(x-a_1)^2$ with a nonzero constant $a_1$.
Thus, using Theorem 4.5 in \cite{Magid}, we state the following Theorem.

\begin{Theorem}
Let $M$ be a Lorentzian hypersurface in the Minkowski space $\mathbb{E}^4_1$ with a nondiagonalizable shape operator 
Then, $(M, g, {\bf x}^T, \lambda)$ is a Ricci soliton if and only if $M$ is an open portion of 
a generalized umbilical Lorentzian hypersurface of $\mathbb{E}^4_1$ given in Example \ref{ex4}. 
\end{Theorem}

\begin{Corollary}
There exists only a shrinking Ricci soliton on a Lorentzian isoparametric hypersurface in $\mathbb{E}^4_1$
with a nondiagonalizable shape operator whose the minimal polynomial has double real roots. 
\end{Corollary}

\begin{Remark}
In the proof of Theorem \ref{thm2}, it can be easily seen that $a_1=0$ for $\lambda=1$. 
Thus, $\mbox{Ric}(X,Y)=0$ for any tangent vectors $X, Y$ of a Lorentzian hypersurface $M$ in $\mathbb{E}^4_1$. 
Moreover, $\mathcal{L}_{{\bf x}^T} g=g$, which means ${\bf x}^T$ is conformal killing vector field. 
Hence, $M$ is an open portion of a generalized cylinder of type I in $\mathbb{E}^4_1$ given in Example \ref{ex5}.
\end{Remark}

\end{document}